\documentclass[a4paper,12pt]{amsart}
\usepackage[english]{babel}
\usepackage[T1]{fontenc}
\usepackage{array}
\usepackage{makecell}
\usepackage[a4paper,margin=2.5cm]{geometry}
\usepackage[justification=centering]{caption}

\usepackage{enumerate}
\usepackage{tabularx} 
\usepackage{amsmath}  
\usepackage{graphicx} 

\usepackage{amsmath,amscd,amsbsy,amssymb,latexsym,url,bm,amsthm,mathrsfs}
\usepackage{epsfig,graphicx,subfigure}
\usepackage{enumerate,balance,mathtools}
\usepackage{wrapfig}
\usepackage{mathrsfs,euscript}
\usepackage[vlined,boxed,commentsnumbered,linesnumbered,ruled]{algorithm2e}
\usepackage{amsfonts}
\usepackage{amsthm}
\usepackage{fancyhdr}
\usepackage{picinpar}
\usepackage{listings}
\usepackage{layout}
\usepackage[all]{xy}
\usepackage{indentfirst}
\usepackage{tikz-cd}
\usepackage{tikz}
\usepackage[all]{xy}
\usepackage{mathrsfs}
\usepackage[colorlinks=true, allcolors=blue]{hyperref}
\usepackage{setspace}

\usepackage{amsmath,amsthm,amssymb,amsxtra,calligra,mathrsfs}
\usepackage{graphicx}
\usepackage{tikz}
\usepackage{tikz-cd} 
\usepackage{xcolor}
\usepackage{upgreek}
\usepackage{comment}
\usepackage{graphicx}
\usepackage{mathtools}
\usepackage{extarrows}
\usepackage{faktor}
\usepackage{mathrsfs,euscript}
\usepackage{comment}
\usepackage{bookmark}
\usepackage{hyperref}
\usepackage{mathrsfs}
\usepackage{multirow}

\hypersetup{
    colorlinks=true, 
    citecolor=blue,
    linkcolor=magenta,
    pdfauthor={},
	bookmarksnumbered=true,
}

\newcommand{\wh}[1]{\widehat{#1}}

\newcommand{\wt}[1]{\widetilde{#1}}

\newcommand{\mb}[1]{\mathbb{#1}}

\newcommand{\ove}[1]{\overline{#1}}

\newcommand{\mtc}[1]{\mathcal{#1}}
\newcommand{\mtf}[1]{\mathfrak{#1}}
\newcommand{\mts}[1]{\mathscr{#1}}

\DeclareMathOperator{\coeff}{coeff}

\DeclareMathOperator{\PGL}{PGL}

\DeclareMathOperator{\ord}{ord}

\DeclareMathOperator{\GIT}{GIT}

\DeclareMathOperator{\vol}{vol}
\DeclareMathOperator{\Gr}{Gr}

\DeclareMathOperator{\sm}{sm}

\DeclareMathOperator{\Bl}{Bl}

\newcommand{\sslash}{\mathbin{\mkern-3mu/\mkern-6mu/\mkern-3mu}}

\newcommand{\sheafHom}{\mathscr{H}\text{\kern -3pt {\calligra\large om}}\,}

\newtheorem{theorem}{Theorem}[section]
\newtheorem{lemma}[theorem]{Lemma}
\newtheorem*{convention}{Convention}
\newtheorem{corollary}[theorem]{Corollary}
\newtheorem{prop}[theorem]{Proposition}
\newtheorem*{ques}{Question}
\newtheorem{proposition}[theorem]{Proposition}

\newtheorem{defn}[theorem]{Definition}

\newtheorem{remark}[theorem]{Remark}

\theoremstyle{remark}

\title{Compactifications of moduli of del Pezzo surfaces via line arrangement and K-stability}
\date{}
\author{Junyan Zhao}
\address{851 S Morgan St, 60607, Chicago, Illinois, USA}
\email{jzhao81@uic.edu}

\begin{document}
\maketitle

\begin{abstract}
In this paper, we study compactifications of the moduli of smooth del Pezzo surfaces using K-stability and the line arrangement. We construct K-moduli of log del Pezzo pairs with sum of lines as boundary divisors, and prove that for $d=2,3,4$, these K-moduli of pairs are isomorphic to the K-moduli spaces of del Pezzo surfaces. For $d=1$, we prove that they are different by exhibiting some walls.

\end{abstract}

\tableofcontents

\section{Introduction}

The moduli spaces of del Pezzo surfaces have been extensively studied thanks to the development of different approaches to construct moduli spaces. 

A weighted marked del Pezzo surface is a pair $(S,B)$ with $S$ a del Pezzo surface and $B=\sum b_iB_i$ a boundary divisor, where $b_i\in[0,1]$ and $B_i$'s are all the finitely many lines on $S$. The KSBA moduli spaces parameterizes all such pairs such that $(S,B)$ is semi log canonical (slc) and $K_S+B$ is ample. For example, when $b_i=1$ for each $i$, the moduli of such quartic del Pezzo pairs is isomorphic to the moduli $\ove{M}_{0,5}$ of rational curves with 5 marked points (cf. \cite{HKT09,HKT14}). Another interesting case is when $b_i=b$, where $b$ is a positive number such that $(S,B)$ is close to a Calabi-Yau pair (cf. \cite{HKT09,HKT14,GKS21}).

The development of K-stability provides a moduli theory for log Fano pairs, called K-moduli spaces. The general K-moduli theory was established by a group of people in the past decade (cf. \cite{ABHLX20,BHLLX21,BLX19,BX19,CP21,LWX21,LXZ22,XZ20}). Roughly speaking, for fixed numerical invariants, we have a projective scheme parameterizing the K-polystable log Fano pairs with these invariants. Moreover, in the surface case, there is a wall crossing structure when we vary the coefficients of the boundary divisors (cf. \cite{ADL19,ADL21}). See Section 2 for definitions and more details. We can ask the similar questions in the K-moduli spaces setting: 

\begin{ques}
Fix a degree $d\in\{1,2,3,4\}$ and a rational number $0\leq c<1$. 
\begin{enumerate}[(1)]
    \item Can we describe the K-moduli space $\ove{M}^{K}_{d,c}$ of K-polystable log Fano pairs $(X,cD)$ admitting a $\mb{Q}$-Gorenstein smoothing to a pair $(X_0,cD_0)$, where $X_0$ a smooth del Pezzo surface of anticanonical degree $d$ and $D_0=\sum L_i$ is the sum of lines (i.e. $-K_{X_0}$-degree $1$ rational curves)?
    \item When varying $c$ in the range such that $(X_0,cD_0)$ is log Fano, do the K-moduli spaces of pairs in (1) change? If yes, can we determine all the values $c_i$'s at which the moduli spaces get changed?
\end{enumerate}
\end{ques}

In particular, when $c=0$, the K-moduli space parameterizes all the del Pezzo surfaces of a fixed degree, which admit a K\"{a}hler-Einstein metric (cf. \cite{OSS16,MM20}).

For $c>0$, if a pair $(X,cD)$ is a K-polystable degeneration of pairs $(X_0,c\sum L_i)$, then $D$ is unique by the separatedness of the K-moduli spaces. In particular, when $X_0$ has du Val singularities, the boundary divisor $D$ can be also written as the sum $\sum L_i'$ of lines, where $L_i'$'s do not have to be distinct. 

We prove in this article that for $d=2,3,4$, there are no walls for the K-moduli spaces $\ove{M}^K_{d,c}$, and they are all isomorphic to K-moduli of del Pezzo surfaces. For $d=1$, we exhibit examples to see that there is indeed some walls.

\begin{theorem} \textup{(Corollary \ref{3}, Proposition \ref{7}, Theorem \ref{2})} For any fixed degree $d\in\{2,3,4\}$, there are no walls for the K-moduli $\ove{M}^K_{d,c}$ when $c$ varies from $0$ to the log Calabi-Yau threshold, and they are isomorphic to the K-moduli space of del Pezzo surfaces of degree $d$.
\end{theorem}

\begin{theorem} \textup{(Proposition \ref{8})}
Let $0<c<\frac{1}{240}$ be a number, and $\ove{M}^K_{1,c}$ be the K-moduli space defined above. Then there is a wall $c=c'<\frac{1}{288}$, which is given by the destabilization of the surface pair, where the surface acquires an $A_7$-singularity.
\end{theorem}

\textbf{Acknowledgements} This problem was motivated when talking with Schock Nolan. I would also like to express my gratitude to Yuchen Liu for sharing some primary ideas. The author also want to thank Izzet Coskun and Schock Nolan for many stimulating discussions.

\section{Preliminaries}

\begin{convention}\textup{
In this paper, we work over the field of complex numbers $\mb{C}$. By a surface, we mean a complex connected projective algebraic surface. We keep the notions on singularities of surface pairs the same as \cite[Chapter 2 and Chapter 4]{KM98}}
\end{convention}

For the reader's convenience, we state the definitions and results in simple versions, especially in dimension 2. Most of them can be generalized to higher dimension.

\begin{defn}
Let $X$ be a normal projective variety, and $D$ be an effective $\mb{Q}$-divisor. Then $(X,D)$ is called a \textup{log Fano pair} if $K_X+D$ is $\mb{Q}$-Cartier and $-(K_X+D)$ is ample. A normal projective variety $X$ is called a \textup{$\mb{Q}$-Fano variety} if $(X,0)$ is a klt log Fano pair.
\end{defn}

\begin{defn}
Let $(X,D)$ be an n-dimensional log Fano pair, and $E$ a prime divisor on a normal projective variety $Y$, where $\pi:Y\rightarrow X$ is a birational morphism. Then the \textup{log discrepancy} of $(X,E)$ with respect to $E$ is $$A_{(X,D)}(E):=1+\coeff_{E}(K_Y-\pi^{*}(K_X+D)).$$ We define the \textup{S-invariant} of $(X,D)$ with respect to $E$ to be $$S_{(X,D)}(E):=\frac{1}{(-K_X-D)^n}\int_{0}^{\infty}\vol_Y(\pi^{*}(-K_X-D)-tE)dt,$$ and the \textup{$\beta$-invariant} of $(X,D)$ with respect to $E$ to be $$\beta_{(X,D)}(E):=A_{(X,D)}(E)-S_{(X,D)}(E)$$
\end{defn}

The first definition of K-(poly/semi)stability of log Fano pairs used test configurations. We refer the reader to \cite{ADL19,Xu21}. There is an equivalent definition using valuations, which is called \emph{valuative criterion} for K-stability. The advantage of this definition is that it is easier to check.

\begin{theorem}\textup{(cf. \cite{Fuj19,Li17,BX19})} A log Fano pair $(X,D)$ is 
\begin{enumerate}
    \item K-semistable if and only if $\beta_{(X,D)}(E)\geq 0$ for any prime divisor $E$ over $X$;
    \item K-stable if and only if $\beta_{(X,D)}(E)>0$ for any prime divisor $E$ over $X$.
\end{enumerate}

\end{theorem}

The following powerful result is called \emph{interpolation} of K-stability. We only state a version that we will use later. For a more general statement, see for example \cite[Proposition 2.13]{ADL19} or \cite[Lemma 2.6]{Der16}.

\begin{theorem}\label{1} 
Let $X$ be a K-semistable $\mb{Q}$-Fano variety, and $D\sim_{\mb{Q}}-rK_X$ be an effective divisor. 
\begin{enumerate}[(1)]
    \item If $(X,\frac{1}{r}D)$ is klt, then $(X,cD)$ is K-stable for any $c\in(0,\frac{1}{r})$;
    \item If $(X,\frac{1}{r}D)$ is log canonical, then $(X,cD)$ is K-semistable for any $c\in(0,\frac{1}{r})$.
\end{enumerate}
\end{theorem}

\begin{theorem}\label{4}
Let $(X,D)$ be a klt log Fano pair which isotrivially degenerates to a K-semistable log Fano pair $(X_0,D_0)$. Then $(X,D)$ is also K-semistable.
\end{theorem}

This follows immediately from the openness of K-(semi)stability (cf. \cite{BLX19,Xu20}), which is useful in our analysis of K-moduli spaces.

Recall that the volume of a divisor $D$ on an $n$-dimensional normal projective variety $Y$ is $$\vol_Y(D):=\lim_{m\to \infty}\frac{\dim H^0(Y,mD)}{m^n/n!}.$$ The divisor $D$ is big by definition if and only if $\vol_Y(D)>0$.

\begin{defn}
Let $x\in (X,D)$ be an n-dimensional klt singularity. Let $\pi:Y\rightarrow X$ be a birational morphism such that $E\subseteq Y$ is an exceptional divisor whose center on $X$ is $\{x\}$. Then the \textup{volume} of $(x\in X)$ with respect to $E$ is $$\vol_{x,X,D}(E):=\lim_{m\to \infty}\frac{\dim\mtc{O}_{X,x}/\{f\in \mtc{O}_{X,x}:\ord_E(f)\geq m\}}{m^n/n!},$$ and the \textup{normalized volume} of $(x\in X)$ with respect to $E$ is $$\wh{\vol}_{x,X,D}(E):=A_{(X,D)}(D)\cdot\vol_{x,X,D}(E).$$ We define the \textup{local volume} of $x\in(X,D)$ to be $$\wh{\vol}(x,X,D):=\inf_{E}\wh{\vol}_{x,X,D}(E),$$ where $E$ runs through all the prime divisor over $X$ whose center on $X$ is $\{x\}$.
\end{defn}

\begin{theorem}\label{6}
\textup{(cf. \cite{Fuj18,LL19,Liu18})}
Let $(X,D)$ be an n-dimensional K-semistable log Fano pair. Then for any $x\in X$, we have $$(-K_X-D)^n\leq \left(1+\frac{1}{n}\right)^n\widehat{\vol}(x,X,D).$$
\end{theorem}

Now let us briefly review some results on the K-moduli spaces of log Fano pairs. We mainly state the results in our setting. For more general statements, see \cite[Theorem 3.1, Remark 3.25]{ADL19} or \cite[Theorem 2.21]{ADL21}.

\begin{defn}
Let $f:(\mtc{X},\mtc{D})\rightarrow B$ be a proper flat morphism to a reduced scheme with normal, geometrically connected fibers of pure dimension $n$, where $\mtc{D}$ is an effective relative Mumford $\mb{Q}$-divisor (cf. \cite[Definition 1]{Kol19}) on $\mtc{X}$ which does not contain any fiber of $f$. Then $f$ is called a \textup{$\mb{Q}$-Gorenstein flat family of log Fano pairs} if $-(K_{\mtc{X}/B}+\mtc{D})$ is $\mb{Q}$-Cartier and ample over $B$.
\end{defn}

\begin{defn}
Let $0<c<1/r$ be a rational number and $(X,cD)$ be a log Fano pair such that $D\sim -rK_X$. A $\mb{Q}$-Gorenstein flat family of log Fano pairs $f:(\mtc{X},c\mtc{D})\rightarrow C$ over a pointed smooth curve $(0\in C)$ is called a \textup{$\mb{Q}$-Gorenstein smoothing} of $(X,D)$ if  
\begin{enumerate}[(1)]
    \item the divisors $\mtc{D}$ and $K_{\mtc{X}/C}$ are both $\mb{Q}$-Cartier, $f$-ample, and $\mtc{D}\sim_{\mb{Q},f}-rK_{\mtc{\mtc{X}/B}}$;
    \item both $f$ and $f|_{\mtc{D}}$ are smooth over $C\setminus\{0\}$, and
    \item $(\mtc{X}_0,c\mtc{D}_0)\simeq (X,cD)$.
\end{enumerate}
\end{defn}

\begin{theorem} \textup{(cf. \cite{ADL19})}
Let $\chi$ be the Hilbert polynomial of an anti-canonically polarized smooth del Pezzo surface $X$ of degree $d$. Let $r$ be the positive integer, and $c\in(0,1/r)$ be a rational number. Consider the moduli pseudo-functor sending a reduced base $S$ to

\[
\mtc{M}^K_{d,c}(S)=\left\{(\mtc{X},\mtc{D})/S\left| \begin{array}{l}(\mtc{X},c\mtc{D})/S\textrm{ is a $\mb{Q}$-Gorenstein smoothable log Fano family},\\
~\mtc{D}\sim_{S,\mtc{Q}}-rK_{\mtc{X}/S}, \textrm{each fiber $(\mtc{X}_s,c\mtc{D}_s)$ is K-semistable, and}\\
\textrm{$\chi(\mtc{X}_s,\mtc{O}_{\mtc{X}_s}(-mK_{\mtc{X}_s}))=\chi(m)$ for $m$ sufficiently divisible.}\end{array}\right.\right\}.
\]
Then there is a smooth quotient stack $\mtc{M}^K_{d}(c)$ of a smooth scheme by a projective general linear group which represents this pseudo-functor. The $\mb{C}$-points of $\mtc{M}^K_{d}(c)$ parameterize K-semistable $\mb{Q}$-Gorenstein smoothable log Fano pairs $(X,cD)$ with Hilbert polynomial $\chi(X,\mtc{O}_X(-mK_X))=\chi(m)$ for sufficiently divisible $m\gg 0$ and $D\sim_{\mb{Q}}-rK_X$. Moreover, the stack $\mtc{M}^K_{d}(c)$ admits a good moduli space $\ove{M}^K_{d}(c)$, which is a normal projective reduced scheme of finite type over $\mb{C}$, whose $\mb{C}$-points parameterize K-polystable log Fano pairs.
\end{theorem}

Let $X$ be a smooth del Pezzo surface of degree $d\in\{1,2,3,4\}$, $L_i$'s be the lines on $X$, and $r$ be the integer such that $\sum L_i\sim -rK_X$.

\begin{defn}
    Let $\mtc{M}^K_{d,c}$ be the stack-theoretic closure of the locally closed substack of $\mtc{M}^K_d(c)$ parameterizing pairs $(X,cD)$, where $X$ is a smooth del Pezzo surface and $D=\sum L_i$, and $\ove{M}^K_{d,c}$ be its good moduli space.  
\end{defn}

\begin{remark}
    \textup{The good moduli space $\ove{M}^K_{d,c}$ is exactly the closed subscheme of $\ove{M}^K_d(c)$, which is the scheme-theoretic closure of the locus parameterizing smooth del Pezzo surfaces with the sum of lines.}
\end{remark}

\begin{theorem}\label{12} \textup{(cf. \cite[Theorem 1.2]{ADL19})} Keep the notation as in the last theorem. There are rational numbers $$0=c_0<c_1<c_2<\cdots<c_n=\frac{1}{r}$$ such that for every $0\leq j<n$, the K-moduli stacks $\mtc{M}^K_{d,c}$ are independent of the choice of $c\in(c_j,c_{j+1})$. Moreover, for every $0\leq j<n$ and $0<\varepsilon\ll1$, one has open immersions $$\mtc{M}^K_{d,c_j-\varepsilon}\hookrightarrow \mtc{M}^K_{d,c_j}\hookleftarrow \mtc{M}^K_{d,c_j+\varepsilon},$$ which descend to projective birational morphisms $$\ove{M}^K_{d,c_j-\varepsilon}\rightarrow \ove{M}^K_{d,c_j}\leftarrow \ove{M}^K_{d,c_j+\varepsilon}.$$

\end{theorem}

\section{Cubic and quartic del Pezzo pairs}

In this section, we focus on the cases when $d=3,4$. The K-moduli of cubic (resp. quartic) del Pezzo surfaces are isomorphic to the GIT moduli spaces of cubics (resp. quartics) (cf. \cite{OSS16,MM20}). In particular, the K-polystable limits of the smooth ones are still embedded in $\mb{P}^3$ (resp. $\mb{P}^4$) as a cubic surface (resp. complete intersection of quadric hypersurfaces), and they have at worst $A_1$ and $A_2$-singularities (resp. $A_1$-singularities). Thus it makes sense to discuss lines on cubic and quartic del Pezzo surfaces with these mild singularities. The lines on singular cubic surfaces were first studied in \cite{Cay69}: they are the degeneration of the 27 lines on smooth cubic surfaces with \emph{multiplicities}, which are nothing but the number of lines which reduce to a given one on a singular cubic surface.

We will prove that for all $c$ in the Fano region, the K-moduli spaces $\ove{M}^K_{d,c}$ of cubic (and quartic) del Pezzo pairs do not have wall crossings, and the moduli spaces are all isomorphic to the K-moduli space of cubic (resp. quartic) del Pezzo surfaces. The proof will proceed by first showing that for $0<c=\varepsilon\ll1$, we have the desired isomorphism induced by the natural forgetful map, and then showing that $\ove{M}^K_{d,c}\simeq \ove{M}^K_{d,\varepsilon}$.

\subsection{The cubic case}

In \cite{OSS16}, the authors essentially proved that for cubic del Pezzo surfaces, the GIT-(semi/poly)stability is equivalent to the K-(semi/poly)stability. As a consequence, the two moduli spaces are isomorphic. However, the K-moduli space was constructed only in recent years. Although this is well-known to experts, for the reader's convenience, here we state a more recent proof using local-global volumes comparison. 

\begin{theorem}\label{5} \textup{(cf. \cite[Section 4.2]{OSS16})}
A cubic surface is K-semistable if and only if it is GIT-semistable. In particular, the K-moduli space of degree $3$ del Pezzo surfaces is isomorphic to the GIT moduli space of cubic surfaces.
\end{theorem}

\begin{proof}
We know that the Fermat cubic surface is K-stable (cf. \cite{Tia87}). By the openness of the K-stability (cf. Theorem \ref{4}), a general cubic surface is K-stable. Denote by $\ove{M}^K_3$ by the K-moduli space of cubic del Pezzo surfaces, which parameterizes K-stable smooth cubic surfaces and their K-polystable limits. Let $X\in \ove{M}^K_3$ be the limit of a 1-parameter family of $\{X_t\}_{t\in T\setminus\{0\}}$ K-semistable smooth cubics.

As the K-semistable surfaces have klt singularities (cf. \cite{Oda13}), in particular have quotient singularities, we let $(x\in X)\simeq (0\in\mb{A}^2/G_x)$ be a singular point (if it exists). Then by Theorem \ref{6}, we have that $$3=(-K_X)^2\leq \frac{9}{4}\widehat{\vol}(x,X)=\frac{9}{4}\cdot\frac{4}{|G_x|}=\frac{9}{|G_x|},$$ which implies that $|G_x|\leq 3$. If $|G_x|=2$, then $x$ is an $A_1$-singularity. If $|G_x|=3$, then $x$ is either an $A_2$-singularity or a $\frac{1}{3}(1,1)$-singularity. The latter case is ruled out by \cite[Proposition 3.10]{KSB88}. We thus conclude that $X$ has at worst $A_1$- or $A_2$-singularities. It follows from \cite[Section 2]{Fuj90} that $X$ can be embedded by $|-K_X|$ into $\mb{P}^3$ as a cubic surface.

Since for hypersurfaces the K-stability implies GIT-stability (cf. \cite[Theorem 3.4]{OSS16} or  \cite[Theorem 2]{PT09}), then we get an open immersion $\mtc{M}^{K}_3\hookrightarrow \mtc{M}^{\GIT}_3$ of moduli stacks, which descends to a birational and injective morphism $\Phi:\ove{M}^K_3\rightarrow \ove{M}^{\GIT}_3$ between good moduli spaces. Notice that the $\ove{M}^{\GIT}_3$ is normal by the properties of GIT quotients. As both of these two moduli spaces are proper, then $\Phi$ is a finite map, and thus $\Phi$ is an isomorphism by Zariski Main Theorem.

\end{proof}

\begin{remark}
    \textup{We will frequently call $\ove{M}^K_3$ the K-moduli spaces of cubic surfaces, by which we mean the K-moduli compactification of smooth K-polystable cubic surfaces.}
\end{remark}

In the proof of the theorem, we deduce that a K-polystable cubic del Pezzo surface has at worst $A_1$- or $A_2$-singularities. This partially recovers the following classical result of Hilbert.

\begin{theorem} \textup{(cf. \cite{Hil70})}
A cubic surface $X\subseteq\mb{P}^3$ is 
\begin{enumerate}[(i)]
\item GIT-stable if and only if it has at worst $A_1$-singularities;
\item GIT-strictly polystable if and only if it is isomorphic to the cubic $X_0$ defined by $x_0^3=x_1x_2x_3$;
\item GIT-semistable if and only if it has at worst $A_1$- or $A_2$-singularities.
\end{enumerate}
\end{theorem}

For the semistable cubic surfaces, we know how the lines degenerate. In other words, we know the multiplicities of the lines (cf. \cite[Table 1]{Tu05}).

\begin{prop}\label{10} 
Let $X$ be a semistable cubic surface and $L_i$'s be the 27 lines on it. Then the pair $(X,c\sum L_i)$ is log canonical for $0<c<\frac{1}{9}$. 
\end{prop}

\begin{proof}
    It is proven in \cite[Proposition 5.8]{GKS21} that if $X$ the log canonical threshold is larger than $\frac{1}{9}$ if $X$ has at worst $A_1$-singularities. 
    
    Now let us deal with the surfaces with $A_2$-singularities. We first give an explicit construction of the surface $X_0=\{x_0^3=x_1x_2x_3\}$. Let $(x:y:z)$ be the homogeneous coordinates of $\mb{P}^2$, then blow up the tangent vector given by $\{y=0\}$ at $(1:0:0)$, the tangent vector given by $\{z=0\}$ at $(0:1:0)$, and the tangent vector given by $\{x=0\}$ at $(0:0:1)$. Denote by $E_i$ and $F_i$ the exceptional divisors on the blow-up with self-intersection $-2$ and $-1$ respectively, and $H_i$ the $(-2)$-curves in the class $$H-E_1-2F_1-E_2-F_2,\quad H-E_2-2F_2-E_3-F_3,\quad H-E_3-2F_3-E_1-F_1,$$ respectively, where $i=1,2,3$. Denote this surface by $\wt{X}_0$. Finally contract the three pairs of $(-2)$-curves to get $X_0$. Notice that there are four sections of $\omega^{*}_{\wt{X}_0}$ coming from $xyz,y^2z,z^2x,x^2y\in H^0(\mb{P}^2,\omega^{*}_{\mb{P}^2})$. They give rise to a morphism from $\wt{X}_0$ to $\mb{P}^3$ contracting $(-2)$-curves, with image $X_0\subseteq \mb{P}^3$ given by $x_0^3=x_1x_2x_3$.

Let $\pi:\wt{X}_0\rightarrow X_0$ be the contraction map. This surface has three $A_2$-singularities, and the 27 lines degenerate to the images of $F_1,F_2,F_3$, denoted by $G_1$, $G_2$, $G_3$, each of which is of multiplicity $9$ (cf. \cite[Proposition 4.1(ii)]{Tu05}). Then by computing intersection numbers, we deduce that $$\pi^{*}G_1=F_1+\frac{1}{3}E_1+\frac{2}{3}H_1+\frac{1}{3}H_2+\frac{2}{3}E_2,$$ and the other two relations for $\pi^{*}G_2$ and $\pi^{*}G_3$. As a result, we obtain that $$A_{(X_0,c\sum L_j)}(F_i)=A_{(X_0,c\sum L_j)}(E_i)=A_{(X_0,c\sum L_j)}(H_i)=1-9c>0$$ for $0<c<\frac{1}{9}$. Thus the pair $(X_0,c\sum L_j)$ is log canonical for $0<c<\frac{1}{9}$. The other pairs whose surfaces have $A_2$-singularities specially degenerate to this case, so we get the result we desire.
\end{proof}

\begin{proposition}\label{14}
    Let $0<\varepsilon\ll1$ be a rational number. Then the forgetful map $\varphi:\mtc{M}^K_3(\varepsilon)\rightarrow \mtc{M}^K_3$ is proper. 
\end{proposition}

\begin{proof}
    Let $\pi^{\circ}:(\mts{X}^{\circ},c\mts{D}^{\circ})\rightarrow T\setminus\{0\}$ be a family of K-semistable pairs, where $(\mts{X}^{\circ},c\mts{D}^{\circ})_t$ is isomorphic to a smooth del Pezzo surface with the sum of lines, and $\mts{X}\rightarrow T$ be an extension of $\mts{X}^{\circ}\rightarrow T\setminus\{0\}$. It suffices to show that there exists a unique extension $\pi:(\mts{X},c\mts{D})\rightarrow T$ of $\pi^{\circ}$. The uniqueness is apparent to us: as $\mts{X}_0$ is normal, the filling $\mts{D}_0$ must be obtained in the following way if it exists. Let $\mts{X}^{\sm}$ be the open locus of $\mts{X}$ such that  $\pi:\mts{X}\rightarrow T$ is smooth on $\mts{X}^{\sm}$. Taking the scheme-theoretic closure of $\mts{D}^{\circ}$, and restricting it to $\mts{X}^{\sm}_0$, the $\mts{D}_0$ is the extension (by taking closure) of this restricted divisor on $\mts{X}^{\sm}_0$ to $\mts{X}_0$ as a Weil divisor.
    
    Now we only need to display such a filling. Recall that the central fiber $\mts{X}_0$ has at worst $A_2$-singularities. It was displayed in \cite{Cay69} that the lines on general fibers degenerate to lines on $\mts{X}_0$ with multiplicities, denoted by $\sum L_{i,0}$. It follows from Proposition \ref{10} that the pair $(\mts{X}_0,c\sum L_{i,0})$ is log canonical when $c=\frac{1}{9}$. It follows from interpolation that the pair $(\mts{X}_0,c\sum L_{i,0})$ is K-semistable for $0<c\ll1$, and this gives a desired filling.
\end{proof}

\begin{prop}\label{15}
    Let $0<\varepsilon\ll1$ be a rational number. Then there is an isomorphism $\ove{M}^K_{3,\varepsilon}\simeq \ove{M}^K_{3}$.
\end{prop}

\begin{proof}
    We claim that the forgetful map $\varphi:\mtc{M}^K_{3,\varepsilon}\rightarrow\mtc{M}^K_3$ is finite. Since $\varphi$ is representable, by Proposition \ref{14}, it suffices to show that it is quasi-finite. In the proof of Proposition \ref{14}, we in fact show that for each K-semistable pair $(X,cD)$ in the stack $\mtc{M}^K_{3,\varepsilon}$, the divisor $D$ is the sum of the lines on $X$ counted with multiplicities. For every cubic surface $X$ with at worst $A_2$-singularities, there are only finitely many lines on $X$, hence $\varphi$ is quasi-finite. The finite forgetful map $\varphi$ descends to a finite morphism between good moduli spaces $\psi:\ove{M}^K_{3,\varepsilon}\rightarrow \ove{M}^K_3$. As the K-moduli $\ove{M}^K_3$ is normal, and the morphism $\psi$ is birational and finite, it follows from the Zariski's Main Theorem that $\psi$ is an isomorphism.
\end{proof}

\begin{remark}
    \textup{In fact, the forgetful map $\varphi$ between stacks is an isomorphism: for the universal family $\mts{X}$ over $\mtc{M}^K_3$, as we can still define lines with multiplicities on mildly singular cubic surfaces, there is a divisor $\mts{D}$ on $\mts{X}$ such that each fiber $(\mts{X},\mts{D})_t$ is a cubic with sum of lines. This gives the inverse morphism of $\phi$.}
\end{remark}

\begin{corollary}\label{3}
Let $X$ be a K-semistable cubic surface and $L_i$'s are the 27 lines on it. Then the pair $(X,c\sum L_i)$ is K-semistable for $0<c<\frac{1}{9}$. In other words, there are no walls for $\ove{M}^K_{3,c}$ when $c$ varies in $(0,\frac{1}{9})$. We have a natural isomorphism $$\ove{M}^K_{3,c}\stackrel{\sim}{\longrightarrow} \ove{M}^K_{3}$$ induced by the forgetful map.
\end{corollary}

\begin{proof}
    This follows immediately from Proposition \ref{10} and Proposition \ref{15} that there are no walls for the K-moduli spaces $\ove{M}^{K}_{3,c}$, and they are all isomorphic to the K-moduli of cubic del Pezzo surfaces, the isomorphism being induced by the forgetful map $\varphi$.
  
\end{proof}

\subsection{The quartic case}

The existence of K\"{a}hler-Einstein metric on quartic surfaces was studied in \cite{MM20,OSS16}. 

\begin{theorem} \textup{(cf. \cite[Theorem 4.1, 4.2]{OSS16})}
A K-semistable quartic del Pezzo surface has at worst $A_1$-singularities.
\end{theorem}

By \cite[Theorem 3.4]{HW81}, we know that every ADE del Pezzo surface is the contraction of the $(-2)$-curves of a blow-up of $\mb{P}^2$ at points in \emph{almost general position} (cf. \cite[Definition 3.2]{HW81}), that is, in the position such that no $(-k)$-curves will be created under the blow-up for any $k\geq 3$. Also, it follows from \cite[Section 2]{Fuj90} that an ADE del Pezzo surface of degree $4$ can be always anti-canonically embedded in to $\mb{P}^4$ as an complete intersection of two quadric 3-folds. Therefore, similar as in the cubic case, we have a canonical choice of the degeneration of the 16 lines on those ADE del Pezzo quartic surfaces.

\begin{prop}\label{7}
Let $X$ be a K-semistable quartic del Pezzo surface and $L_i$'s the 16 lines on it. Then the pair $(X,c\sum L_i)$ is K-semistable for $0<c<\frac{1}{4}$. Moreover, there is a natural isomorphism $\ove{M}^K_{4,c}\simeq \ove{M}^K_4$ of the K-moduli spaces, induced by the forgetful map $\mtc{M}^K_{4,\varepsilon}\rightarrow \mtc{M}^K_4$. 
\end{prop}

\begin{proof}
As in the proof of Corollary \ref{3}, we have a family $(\mtc{X},\mtc{D})$ over $\mtc{M}^K_4$ such that a general fiber is a smooth quartic with the 16 lines on it. By interpolation of K-stability (cf. Theorem \ref{1}), if we prove that these pairs are log canonical for $0<c\leq\frac{1}{4}$, then $(\mtc{X},c\mtc{D})$ has K-semistable fibers for any $0\leq c<1/4$. In particular, the same argument of the second statement of Corollary \ref{3} shows that $\ove{M}^K_{4,c}\simeq\ove{M}^K_4$ for any $c\in (0,\frac{1}{4})$.

Now we prove that the each fiber in the family $(\mtc{X},c\mtc{D})$ is log canonical for $0\leq c\leq 1/4$. Let $(X,cD=c\sum L_i)$ be an arbitrary fiber. 

\begin{enumerate}[(i)]
\item If $X$ is smooth, then the 16 lines are distinct and $\sum L_i$ is normal crossing. Thus the pair is log canonical automatically.
\item If $X$ has one $A_1$-singularity $p$, then we claim that there exist exactly four lines of multiplicity two and eight lines of multiplicity one, and that these double lines pass through $p$, while the remaining eight lines avoid it. 

First notice that a quintic del Pezzo surface $X$ with exactly one $A_1$-singularity is obtained by contracting the $(-2)$-curve on the blow-up of $\mb{P}^2$ along three distinct points $p_1,p_2,p_3$ and a tangent vector supported at $p_4$. In fact, let $p\in X$ be the singularity, and $\wt{X}\rightarrow X$ the blow-up of $X$ at $p$ with an exceptional divisor, which is a $(-2)$-curve. By \cite[Theorem 3.4]{HW81}, 
the surface $\wt{X}$ is obtained from blowing up $\mb{P}^2$, thus there must be a $(-1)$-curve $C$ intersecting $E$. Contracting first $C$, then the proper transform of $E$, one get a smooth del Pezzo surface of degree $6$, which is a blow-up of $\mb{P}^2$ at three distinct points.

Consider a one-parameter family of projective planes $\mb{P}^2\times\mb{A}^1\rightarrow\mb{A}^1$, and let $l_1,...,l_5$ be five sections of $\mb{A}^1$ such that over $0\neq t\in \mb{A}^1$, $l_1,...,l_5$ does not intersect, while over $0$, only $l_4$ and $l_5$ intersect transversely. Blowing up $\mb{P}^2\times\mb{A}^1$ along $l_1\cup\cdots\cup l_5$ with reduced scheme structure, one get a degeneration $\mtf{X}\rightarrow \mb{A}^1$ of smooth quartic del Pezzo surfaces to a singular one with exactly one $A_1$-singlarity. Denote $p_i(t)$ the intersection of $l_i$ with the fiber $\mb{P}^2_t$ over $t\in \mb{A}^1$. Then under the family $\mtf{X}\rightarrow \mb{A}^1$, the lines $L(p_i(t),p_4(t))$ and $L(p_i(t),p_5(t))$ (for i=1,2,3) on general fibers $\mtf{X}_t$ degenerate to the same line $L(p_i(0),p_4(0))=L(p_i(0),p_5(0))$, and the exceptional divisors over $p_4(t)$ and $p_5(t)$ also degenerate to the same line, which is the exceptional divisor over $p_4(0)=p_5(0)$. The other eight lines over general fibers degenerate to distinct lines on $\mtf{X}_0$. Finally observe that $-K_{\mtf{X}/\mb{A}^1}$ gives rise to an embedding of $\mtf{X}$ into $\mb{P}^4\times\mb{A}^1$ over $\mb{A}^1$. Thus this degeneration indeed occurs in the Hilbert scheme.

Let $E$ be the exceptional divisor of the blow-up $\pi:\wt{X}\rightarrow X$. Then $E+\pi^{*}\sum L_i$ has simple normal crossing support. Moreover, we have that $$A_{(X,c\sum L_i)}(E)=1-4c>0$$ for any $0<c<\frac{1}{4}$. As the multiplicity of the proper transform of $L_i$ is at most two, then the pair is log canonical.

\item If $X$ has two $A_1$-singularities $p$ and $q$, then there exist one line of multiplicity four, four lines of multiplicity two and four lines of multiplicity one. This follows from the same argument as in (i), and in fact the line with multiplicity four passes through both $p$ and $q$, the line with multiplicity two passes through either $p$ or $q$, and the line with multiplicity one avoids both $p$ and $q$. For each singularity, there are exactly eight lines (counted with multiplicities) passing through it. Let $E$ be the exceptional divisor of the blow-up at $p$. Then we have that $$A_{(X,c\sum L_i)}(E)=1-4c>0$$ for any $0<c<\frac{1}{4}$. As the multiplicity of the proper transform of $L_i$ is at most four, then the pair is log canonical. 

The proof of cases (iii) and (iv) are completely the same, we will omit part of the details.

\item Suppose that $X$ has three $A_1$-singularities. Then $X$ is obtained as follows: blow up $\mb{P}^2$ at a general tangent vector and at three curvilinear points on a general line, where two of them collide, then take the ample model. There exist two lines of multiplicity four, and four lines of multiplicity two. For each singularity, there are exactly eight lines (counted with multiplicities) passing through it. For the same reason as in (ii), the pair is log canonical.
\item If $X$ has four $A_1$-singularities, then it is obtained as follows: blow up $\mb{P}^2$ at a general point $P$ and at two tangent vectors whose supporting lines pass through $P$, then take the ample model. There are four lines of multiplicity four. For each singularity, there are exactly eight lines (counted with multiplicities) passing through it. For the same reason as in (ii), the pair is log canonical.
\end{enumerate}
\end{proof}

In the proof of the Proposition \ref{7}, we can deduce the following result, which is proved in \cite{Tu05} for cubic surface case. The explicit equations of the lines on any $A_1$ cubic del Pezzo surface can be carried out from the normal form as a cubic hypersurface in $\mb{P}^3$. When reducing from the smooth surface to a singular one, the 27 lines on a smooth surface reduce to the lines on the corresponding singular surface. The \emph{multiplicity} (cf. \cite{Cay69}) of a line $l$ of a singular surface is nothing but the number of lines which reduce to $l$.

\begin{corollary}\label{13} 
    Let $X\subseteq \mb{P}^4$ be a quartic del Pezzo surface with at worst $A_1$-singularities, and $l$ is a line on it. Then 
    \begin{enumerate}
        \item If $l$ does not contain any singular point, then $l$ is of multiplicity 1;
        \item If $l$ contains exactly one singularity, then $l$ is of multiplicity 2;
        \item If $l$ contains two singularities, then $l$ is of multiplicity 4.
    \end{enumerate}
\end{corollary}

\begin{remark}
\textup{Using the same argument as in Theorem \ref{5}, one can prove that the K-moduli space of quartic del Pezzo surfaces is isomorphic to the GIT-moduli space $$\Gr(2,H^0(\mb{P}^4,\mtc{O}_{\mb{P}^4}(2)))^{ss}\sslash \PGL(5).$$ See \cite{OSS16,SS17} for details. As a consequence, each K-moduli space $\ove{M}^K_{4,c}$ is isomorphic to this GIT-moduli space.}
\end{remark}

\section{Degree two case}

In this section, we prove that there are no walls for the K-moduli spaces $\ove{M}^K_{2,c}$ when we vary the coefficient $c$ from $0$ to $\frac{1}{28}$.

Recall that a smooth del Pezzo surface of degree $2$ is a double cover of $\mb{P}^2$ branched along a quartic curve, and the 56 lines are sent pairwise to the 28 bitangent lines of the quartic. Thus the K-stability of a degree $2$ del Pezzo surface $X$ is equivalent to the K-stability of a del Pezzo pair $(\mb{P}^2,\frac{1}{2}C_4)$ (cf. \cite[Remark 6.2]{ADL19}), where $C_4$ is the quartic plane curve along which the double cover $X\rightarrow \mb{P}^2$ is branched.

In \cite{OSS16}, the authors give a description of the K-moduli space of del Pezzo surfaces of degree $2$. It is diffeomorphic to the blow-up of $\mb{P}H^0(\mb{P}^2,\mtc{O}_{\mb{P}^2}(4))^{ss}\sslash\PGL(3)$ at the point parameterizing the double conic. Moreover, each point $[s]$ on the exceptional divisor $E$ represents a surface which is a double cover of $\mb{P}(1,1,4)$ branched along a hyperelliptic curve $z^2=f_8(x,y)$, where $f_8$ is a GIT-polystable octic binary form.  In \cite{ADL19}, the authors study the wall crossing of the K-moduli of pairs $(\mb{P}^2,cC_4)$ when $c$ varies from $0$ to $\frac{3}{4}$. They proved that there is a unique wall $c=\frac{3}{8}$. As a result, the K-moduli of degree $2$ del Pezzo surfaces is isomorphic to a weighted blow-up of the GIT moduli space $\mb{P}H^0(\mb{P}^2,\mtc{O}_{\mb{P}^2}(4))^{ss}\sslash\PGL(3)$ at the point parameterizing the double conic. For a rigorous proof, we refer the reader to \cite[Proposition 6.12]{ADL21}, where the authors prove that the K-moduli of the quartic double solids is isomorphic to the K-moduli space of K-polystable pairs $(X,\frac{1}{2}D)$ which admit a 
$\mb{Q}$-Gorenstein smoothing to $(\mb{P}^3,\frac{1}{2}S)$ with $S\in|-K_{\mb{P}^3}|$ a smooth quartic K3 surface.

Observe that the degeneration of $\mb{P}^2$ to $\mb{P}(1,1,4)$ can occur in $\mb{P}(1,1,1,2)$ (cf. \cite[Theorem 5.14]{ADL19}), thus there is a canonical choice of the degeneration of the curves and the bitangent lines, and the multiplicity is well-defined. Therefore, we can apply the same approach as in the degree $d=3,4$ cases, to show that there are no walls. By interpolation of K-semistability, it suffices to check that the $(X,\frac{1}{28}\sum E_i)$ is log canonical, where $X$ is a K-polystable degree quadric del Pezzo surface and $E_i$'s are the lines on it. It further reduces to checking that $(\mb{P}^2,\frac{1}{2}C+\frac{1}{28}\sum L_i)$ or $(\mb{P}(1,1,4),\frac{1}{2}C+\frac{1}{28}\sum L_i)$ is log canonical (cf. \cite[Proposition 5.20]{KM98}), where $C$ is the branched curve and $L_i$'s are the bitangent lines of it. 

The classification of the semistable plane quartics is well-known to the experts (cf. \cite{MFK94}). For the reader's convenience, we state the result here.

\begin{lemma} \textup{(cf. \cite[Theorem 2]{HL10})}
Let $G=\PGL(3)$ act on the space of plane quartic. A plane quartic curve $C_4\subseteq\mb{P}^2$ is  
\begin{enumerate}[(i)]
  \item stable if and only if it has at worst $A_1$ or $A_2$ singularities;
  \item strictly semistable if and only it is a double conic or has a tacnode. Moreover, $C_4$ belongs to a minimal orbit if and only if it is either a double conic or the union of two tangent conics, where at least one is smooth.
\end{enumerate}
\end{lemma}

\begin{remark}
\textup{In the case (ii), if both of the (distinct) conics are smooth, then we call it \emph{cateye}; if there is a singular one, we call it \emph{ox}. These are the only two types of polystable quartics with infinite stabilizers.}
\end{remark}

\subsection{Singularities of plane curves with infinite stabilizers}

Let us first compute the log canonical property for the special cases where the quartics have infinite stabilizers. 

\begin{lemma}
The pair $(\mb{P}^2,\frac{1}{2}(C_2+C'_2)+c\sum L_i)$ is log canonical for $0<c<\frac{1}{28}$, where $C_2$ and $C'_2$ are two smooth conics tangential at two points $p$ and $q$.
\end{lemma}

\begin{proof}
First observe that the $(\mb{P}^2,C_2+C'_2)$ is the degeneration (in the Hilbert scheme) of pairs $(\mb{P}^2,C_2(t)+C'_2(t))_{t\in T}$ where the boundary divisors consist of two smooth conics with a tacnode and two nodes. It follows from \cite[Section 3.4 case 2]{CS03} that the arrangement of the 28 bitangent lines of $C_2+C'_2$ is $\sum L_i=6L_p+6L_q+16L_{pq}$, where $L_p$ and $L_q$ are tangent lines of the conics at $p$ and $q$ respectively, and $L_{pq}$ is the line connecting $p$ and $q$. It follows immediately that the coefficients of $\frac{1}{2}(C_2+C'_2)+c\sum L_i$ are all smaller than $1$.

Taking the minimal log resolution of $(\mb{P}^2,\frac{1}{2}(C_2+C'_2)+c\sum L_i)$. By symmetry, we only need to look at the log resolution at $q$. Let $E$ and $F$ be the exceptional divisors over $q$ with self-intersection $-2$ and $-1$ respectively. Then one has $A(E)=1-22c$ and $A(F)=1-28c$, which are positive when $0<c<\frac{1}{28}$.
\end{proof}

The ox case is similar.

\begin{lemma}
The pair $(\mb{P}^2,\frac{1}{2}(C_2+M_1+M_1')+c\sum L_i)$ is log canonical for $0<c<\frac{1}{28}$, where $C_2$ is a smooth conic, and $M_1,M_1'$ are two distinct lines tangential to $C_2$ at two points $p$ and $q$ respectively.
\end{lemma}

\begin{proof}
First observe that the $(\mb{P}^2,C_2+M_1+M_1')$ is the degeneration (in the Hilbert scheme) of pairs $(\mb{P}^2,C_2(t)+M_1(t)+M_1'(t))_{t\in T}$ where the boundary divisors consist of two distinct lines $M_1(t)$, $M_1'(t)$ and a smooth conic $C_2(t)$ tangent to $M_1(t)$ at a point $p(t)$ and meeting $M_1'(t)$ transversely. It follows from \cite[Section 3.4 case 10]{CS03} that the arrangement of the 28 bitangent lines of $C_2+M_1+M_1'$ is $\sum L_i=6M_1+6M_1'+16L_{pq}$, where $L_{pq}$ is the line connecting $p$ and $q$. It follows immediately that the coefficients of $\frac{1}{2}(C_2+C'_2)+c\sum L_i$ are all smaller than $1$.

Taking the minimal log resolution of $(\mb{P}^2,\frac{1}{2}(C_2+M_1+M_1')+c\sum L_i)$. By symmetry, we only need to look at the log resolution at $q$. Let $E$ and $F$ be the exceptional divisors over $q$ with self-intersection $-2$ and $-1$ respectively. Then one has $A(E)=1-22c$ and $A(F)=1-28c$, which are positive when $0<c<\frac{1}{28}$.
\end{proof}

\subsection{Singularities of reducible plane curves with finite stabilizers}

Now we consider the reducible quartics with finite stabilizers. 

\begin{prop}
    Let $C_4$ be a reducible quartic curve in $\mb{P}^2$ with finite stabilizers under the $\PGL(3)$-action, and $L_i$'s are the degeneration on $C_4$ of the 28 bitangent lines of smooth quartics. Then $(\mb{P}^2,\frac{1}{2}C_4+c\sum L_i)$ is log canonical for $0\leq c\leq \frac{1}{28}$.
\end{prop}

\begin{proof}
    The classification of $C_4$ is listed in \cite[Section 3.4]{CS03}. For each class, we can run the same argument. For simplicity, we only prove the statement for one class where $C_4=M_1+M_2+M_3+M_4$ is the union of four general lines (forming three pairs of nodes). In this case, we have that $\sum L_i=4(M_1+\cdots+M_7)$, where $M_5,M_6,M_7$ are the lines joining the three pairs of nodes (cf. \cite[Section 3.4 case 11]{CS03}). Then the blow-up of the pair $(\mb{P}^2,\frac{1}{2}C_4+c\sum L_i)$ at the six nodes (with exceptional divisors $E_1,...,E_6$) is log smooth. As the boundary divisors $\frac{1}{2}C_4+c\sum L_i$ have coefficients less than $1$ and $A_{(\mb{P}^2,\frac{1}{2}C_4+c\sum L_i)}(E_i)=1-12c>0$, then the pair we consider is log canonical.
\end{proof}

\subsection{Singularities of irreducible plane curves with finite stabilizers}

Now we deal with the general case: an irreducible plane quartic $C_4$ with at worst $A_1$ or $A_2$ singularities. We first assume that $C_4$ is irreducible. Then the \cite[Table 3.2, Lemma 3.3.1]{CS03} describe how these bitangent lines degenerate. We claim that all these curves give log canonical pairs $(\mb{P}^2,\frac{1}{2}C_4+c\sum L_i)$. 

\begin{lemma}
The pair $(\mb{P}^2,\frac{1}{2}C_4+c\sum L_i)$ is log canonical for $0<c<\frac{1}{28}$, where $C_4$ has three nodes.
\end{lemma}

\begin{proof}
We know from \cite[Lemma 3.3.1]{CS03} that the 28 lines degenerate to 4 bitangent lines of multiplicity 1 not passing through the nodes, 6 tangent lines of multiplicity 2 passing through exactly one node, and 3 lines of multiplicity 4 containing two nodes. Thus for each node $p$, there are 12 lines passing through it, and at most 4 lines tangential to it. One can blow up at $p$ twice to get a log resolution. We have $A(E)=1-16c>0$ and $A(F)\geq \frac{3}{2}-20c>0$ for $0<c<1/28$, where $E$ and $F$ are the two exceptional divisors over $p$ with self-intersection $-2$ and $-1$ respectively. The singularities at other points are milder, thus $(\mb{P}^2,\frac{1}{2}C_4+c\sum L_i)$ is log canonical as desired.

\end{proof}

\begin{lemma}
The pair $(\mb{P}^2,\frac{1}{2}C_4+c\sum L_i)$ is log canonical for $0<c<\frac{1}{28}$, where $C_4$ has three cusps.
\end{lemma}

\begin{proof}
We know from \cite[Lemma 3.3.1]{CS03} that the 28 lines degenerate to 4 bitangent lines of multiplicity 1 not passing through the cusps, and 3 lines of multiplicity 9 containing two cusps. Thus for each cusp $p$, there are 18 lines passing through it, and none of them are tangential to it. One can blow up at $p$ three times to get a log resolution. Let $E,F,G$ be the exceptional divisors with self-intersection $-3,-2,-1$ respectively. We have $A(E)=1-18c>0$, $A(F)=\frac{3}{2}-18c>0$, and $A(G)=2-36c>0$ for $0<c<1/28$. The singularities at other points are milder, thus $(\mb{P}^2,\frac{1}{2}C_4+c\sum L_i)$ is log canonical as desired.

\end{proof}

\begin{lemma}
The pair $(\mb{P}^2,\frac{1}{2}C_4+c\sum L_i)$ is log canonical for $0<c<\frac{1}{28}$, where $C_4$ has one node and two cusps.
\end{lemma}

\begin{proof}
By \cite[Lemma 3.3.1]{CS03}, we see that the 28 lines degenerate to one bitangent line of multiplicity 1 not passing through the singularities, 2 tangent lines of multiplicity 3 passing through exactly one cusp, 1 line of multiplicity 9 containing two cusps, and 2 lines of multiplicity 6 passing through one node and one cusp. 

Thus for the node $p$, there are 12 lines passing through it, and none of them are tangential to it. One can blow up at $p$ to get a log resolution and the log discrepancy with respect to the exceptional divisor is $1-12c>0$. 

For each cusp $q$, there are 18 lines passing through it, and at most 3 lines tangential to it. Resolving $q$ by blowing up three times, one gets three exceptional divisors $E,F,G$ with self-intersection $-3,-2,-1$ respectively. We have $A(E)=1-18c>0$, $A(F)\geq \frac{3}{2}-21c>0$, and $A(G)\geq 2-39c>0$ for $0<c<1/28$. The singularities at other points are milder, thus $(\mb{P}^2,\frac{1}{2}C_4+c\sum L_i)$ is log canonical as desired.

\end{proof}

\begin{lemma}
The pair $(\mb{P}^2,\frac{1}{2}C_4+c\sum L_i)$ is log canonical for $0<c<\frac{1}{28}$, where $C_4$ has one cusp and two nodes.
\end{lemma}

\begin{proof}
By \cite[Lemma 3.3.1]{CS03}, we see that the 28 lines degenerate to 2 bitangent lines of multiplicity 1 not passing through the singularities, 2 tangent lines of multiplicity 3 passing through exactly the cusp, 2 tangent lines of multiplicity 2 passing through exactly one node, 1 line of multiplicity 4 containing two nodes, and 2 lines of multiplicity 6 passing through one node and one cusp. 

Thus for the cusp $p$, there are 18 lines passing through it, and at most 6 of them are tangential to it. Resolving $p$ by blowing up three times, one gets three exceptional divisors $E,F,G$ with self-intersection $-3,-2,-1$ respectively. We have $A(E)=1-18c>0$, $A(F)\geq \frac{3}{2}-24c>0$, and $A(G)\geq 2-42c>0$ for $0<c<1/28$. 

For each node $q$, there are 12 lines passing through it, and at most 2 lines tangential to it. One can blow up at $p$ twice to get a log resolution. We have $A(E')=1-12c>0$ and $A(F')\geq \frac{3}{2}-14c>0$ for $0<c<1/28$, where $E'$ and $F'$ are the two exceptional divisors over $p$ with self-intersection $-2$ and $-1$ respectively. The singularities at other points are milder, thus $(\mb{P}^2,\frac{1}{2}C_4+c\sum L_i)$ is log canonical as desired.

\end{proof}

The other pairs with $C_4$ irreducible have milder singularity than the four types in the above four lemmas. Applying the same argument to all these pairs, we conclude the following.

\begin{proposition}
Let $C_4$ be a GIT-semistable plane quartic which is not a double conic. Then the pair $(\mb{P}^2,\frac{1}{2}C_4+c\sum L_i)$ is log canonical for $0<c<\frac{1}{28}$. In particular, these pairs are K-semistable for $0<c<1/28$.
\end{proposition}

\subsection{Singularities of curves in $\mb{P}(1,1,4)$}

We first point out an explicit degeneration of $\mb{P}^2$ to $\mb{P}(1,1,4)$ in $\mb{P}(1,1,1,2)$. Let $(x_0:x_1:x_2:x_3)$ be the homogeneous coordinate of $\mb{P}(1,1,1,2)$ and then consider the the hypersurface $$\mtf{X}:=\left\{(t,(x_0:x_1:x_2:x_3)): t\cdot x_3=x_0x_2-x_1^2\right\}\subseteq\mb{A}^1\times \mb{P}(1,1,1,2).$$ The natural projection $\pi:\mtf{X}\rightarrow \mb{A}^1$ gives a special degeneration of $\mb{P}^2$ to $\mb{P}(1,1,4)$.

Recall that an octic binary form $f_8(x,y)$ is GIT-semistable if and only if each of its zeros has multiplicity at most $4$ (cf. \cite{MFK94} Section 4.1). Also notice that if the point $(x:y:z)$ is an singularity of the curve $C=\{z^2=f_8(x,y)\}$, then $z=0$ and $(x:y)$ is a multiple root of $f_8(x,y)$. 

Let $p_i$ be the intersection points of $C$ and $\{z=0\}$, where $i=1,...,8$. Another important observation is that a bitangent line of a plane quartic degenerate under $\pi$ to two rulings of $\mb{P}(1,1,4)$, each of which passes through some $p_i$, and there are $\binom{8}{2}=28$ pairs in total. Thus we only need to focus on the vertex $p$ of $\mb{P}(1,1,4)$ and the singularities of $C$.

\begin{prop}
The pair $(\mb{P}(1,1,4),\frac{1}{2}C+c\sum L_i)$ is log canonical for $0<c<\frac{1}{28}$.
\end{prop}

\begin{proof}
Notice that $C$ does not pass through the vertex $p$. Let $E$ be the exceptional divisor of $\Bl_{p}\mb{P}(1,1,4)\rightarrow \mb{P}(1,1,4)$. Then $A(E)=\frac{1}{2}-14c>0$ for $0<c<\frac{1}{28}$.

Suppose that $p_i=(x:y:0)$ is a singular point of $C$. We may assume that $x=0$ and $y=1$, and locally the equation of $C$ is $z^2=x^t$, where $t\in\{2,3,4\}$. When $t=4$, $p_i$ is a tacnode, and there are 28 lines passing through it. Let $F$ and $G$ be the exceptional divisors of the minimal resolution of $p_i$ with self-intersection $-2$ and $-1$ respectively. Then $A(F)=1-28c>0$ and $A(G)=1-28c>0$ for $0<c<\frac{1}{28}$. This is also true for the cases $t=2$ and $t=3$. 

\end{proof}

To sum up, we conclude the following by the same argument as in Corollary \ref{3}.

\begin{theorem}\label{2}
Let $\ove{M}^K_2$ be the K-moduli space of degree $2$ del Pezzo surfaces. Then there are no walls for the K-moduli stacks $\mtc{M}^K_{2,c}$, and there is an isomorphism $$\ove{M}^K_{2,c}\simeq \ove{M}^K_2$$ for any $0<c<\frac{1}{28}$.  

\end{theorem}

\section{Discussion on degree one case}

In this section, we display some examples to show that there exist some walls for $M^K_{1,c}$ when $c$ varies from $0$ to $\frac{1}{240}$. The following result describes the K-polystable del Pezzo surfaces with at worst ADE singularities. For the analytic description of the K-moduli space $M^K_{1}$, see \cite[Section 5]{OSS16}.

\begin{prop} \textup{(cf. \cite[Section 6.1.2]{OSS16})}
A nodal del Pezzo surface $X$ of degree $1$ is K-polystable if and only if it has either only $A_k$-singularities with $k\leq 7$, or exactly two $D_4$-singularities, and $X$ is not isomorphic to one of the surfaces parameterized by $(a_1:a_2)\in\mb{P}^1$, which are hypersurfaces in $\mb{P}(1,1,2,3)$ defined by the equations $$w^2=a_1z^3+z^2x^2+zy^4+a_2z^2y^2.$$
\end{prop}

\subsection{Degree one del Pezzo with an $A_7$-singularity}

We aim to prove the following statement.

\begin{prop}\label{8}
    Let $Y$ be a del Pezzo surface of degree one with an $A_7$-singularity, and $L_i$'s be the degeneration of the 240 lines. Then the pair $(Y,c\sum L_i)$ is K-unstable for $c\geq \frac{1}{288}.$
\end{prop}

\begin{proof}
Let $Y$ be a degree one del Pezzo surface with an $A_7$-singularity. Then $Y$ is obtained by blowing up $\mb{P}^2$ at eight curvilinear points on a smooth cubic curve supported at a single point and taking the ample model. Let $E_1,...,E_8$ be the eight exceptional divisor on the blow-up $\wt{Y}$ in the order of blow-up, and $H$ the class of the pull-back of $\mtc{O}_{\mb{P}^2}(1)$. Then the proper transform to $\wt{Y}$ of the 240 lines on $Y$ are of the following types:

\begin{enumerate}[(a)]
\item $E_8$, intersecting $E_7$, of multiplicity 8;
\item $H-E_1-2E_2-2E_3-\cdots-2E_8$, intersecting $E_2$, of multiplicity 28;
\item $2H-E_1-2E_2-3E_3-4E_4-5\sum_{i\geq 5} E_i$, intersecting $E_5$, of multiplicity 56;
\item $3H-2E_1-3E_2-\cdots-8E_7-8E_8$, intersecting $E_1$ and $E_7$, of multiplicity 56;
\item $4H-2E_1-4E_2-6E_3-\sum_{i\geq 4}(i+3)E_i$, intersecting $E_3$ and $E_8$, of multiplicity 56;
\item $5H-\sum_{i=1}^62iE_i-13E_7-14E_8$, intersecting $E_6$ and $E_8$, of multiplicity 28;
\item $6H-3E_1-\sum_{i\geq2} (2i+1)E_i$, intersecting $E_1$ and intersecting $E_8$ twice, of multiplicity 8;
\end{enumerate}

Denote by $L_i$ and $l_i$ the line and its image on $Y$ of one of the above types, and $a_i$ be the multiplicity of the line $L_i$, where $i=0,1,...,6$ is the coefficient of $H$ in the class. Let $\pi:\wt{Y}\rightarrow Y$ be the blow-down morphism. Then we have $$\pi^*l_0=E_8+\frac{7}{8}E_7+\frac{6}{8}E_6+\frac{5}{8}E_5+\frac{4}{8}E_4+\frac{3}{8}E_3+\frac{2}{8}E_2+\frac{1}{8}E_1,$$
$$\pi^*l_1=L_1+\frac{1}{4}E_7+\frac{2}{4}E_6+\frac{3}{4}E_5+\frac{4}{4}E_4+\frac{5}{4}E_3+\frac{6}{4}E_2+\frac{3}{4}E_1,$$
$$\pi^*l_2=L_2+\frac{5}{8}E_7+\frac{10}{8}E_6+\frac{15}{8}E_5+\frac{12}{8}E_4+\frac{9}{8}E_3+\frac{6}{8}E_2+\frac{3}{8}E_1,$$
$$\pi^*l_3=L_3+E_7+E_6+E_5+E_4+E_3+E_2+E_1,$$
$$\pi^*l_4=L_4+\frac{3}{8}E_7+\frac{6}{8}E_6+\frac{9}{8}E_5+\frac{12}{8}E_4+\frac{15}{8}E_3+\frac{10}{8}E_2+\frac{5}{8}E_1,$$
$$\pi^*l_5=L_5+\frac{3}{4}E_7+\frac{6}{4}E_6+\frac{5}{4}E_5+\frac{4}{4}E_4+\frac{3}{4}E_3+\frac{2}{4}E_2+\frac{1}{4}E_1,$$
$$\pi^*l_6=L_6+\frac{1}{8}E_7+\frac{2}{8}E_6+\frac{3}{8}E_5+\frac{4}{8}E_4+\frac{5}{8}E_3+\frac{6}{8}E_2+\frac{7}{8}E_1.$$

In particular, one gets that the log discrepancy $$A_{(Y,c\sum a_iL_i)}(E_3)=A_{(Y,c\sum a_iL_i)}(E_4)=A_{(Y,c\sum a_iL_i)}(E_5)=1-288c<0$$ when $\frac{1}{288}<c<\frac{1}{240}$. Thus there is a wall $0<c<\frac{1}{288}$ given by the degeneration of such pairs.

\end{proof}

\subsection{Degree one del Pezzo with two $D_4$-singularities}

\begin{prop}\label{9}
    Let $Z$ be a del Pezzo surface of degree one with exactly two $D_4$-singularities, and $L_i$'s be the degeneration of the 240 lines. Then the pair $(Z,c\sum L_i)$ is K-stable for any $0<c< \frac{1}{240}$.
\end{prop}

\begin{proof}

Let $Z$ be a degree one del Pezzo surface with exactly $D_4$-singularities. Such surfaces are weighted hypersurfaces in $\mb{P}(1,1,2,3)$ and there is a one-parameter family paramaterizing them (cf. \cite[Example 5.19]{OSS16}). For us, a description by blowing up projective plane is more useful, since we need to figure out the degeneration of the 240 lines.

We see that $Z$ can be obtained by blowing up $\mb{P}^2$ in the following way. Fix four distinct points $p_1,p_2,p_3,p_4$ on a line $L$ in $\mb{P}^2$, and another point $p\in\mb{P}^2\setminus L$. Blow up $\mb{P}^2$ at $p_1$ along the tangent direction to $p$ with exceptional divisors $E_1,F_1$; at $p_2$ along the tangent direction to $p$ with exceptional divisors $E_2,F_2$; and at the point $p_3$ with exceptional divisors $E$. Finally blow up the length $3$ 0-dimensional curvilinear subscheme supported on the line $\ove{pp_4}$ and concentrated at $p$ with exceptional divisors $E_3,F_3,G_3$ (see Figure \ref{blow} for the blow-up procedure and Figure \ref{conf} for the configuration of the (-1)-curves and (-2)-curves on $\wt{Z}$). Denote this surface by $\wt{Z}$, and the ample model by $Z$. 

\begin{figure}
\centering 
\includegraphics[width=0.4\textwidth]{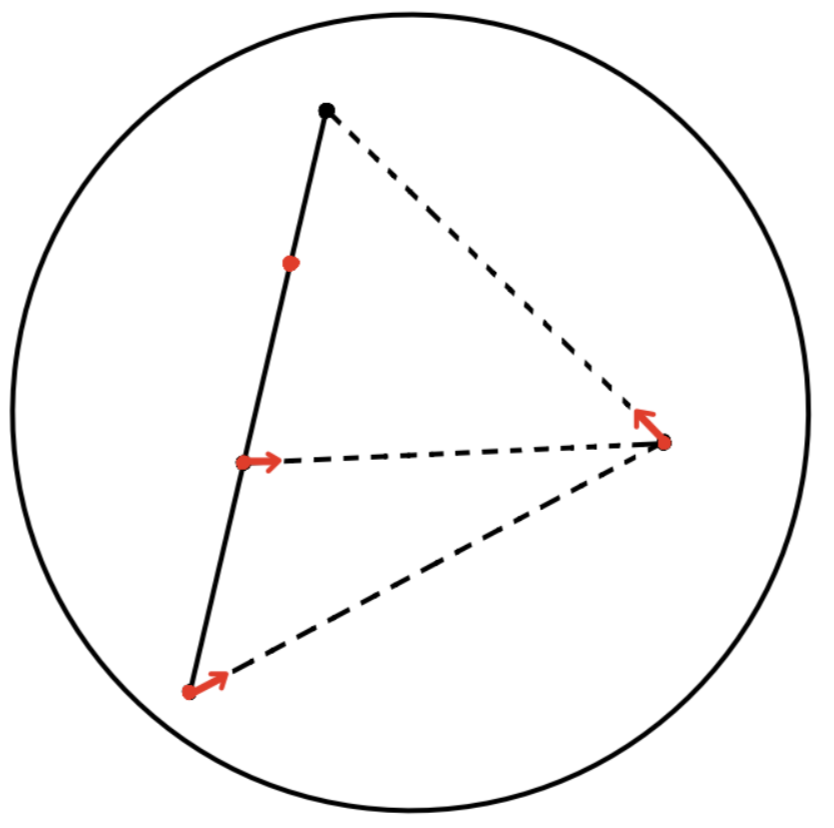} 
\caption{Blow-up of $\mb{P}^2$ to get $\wt{Z}$} 
\label{blow}
\end{figure}

\begin{figure}
\centering 
\includegraphics[width=0.8\textwidth]{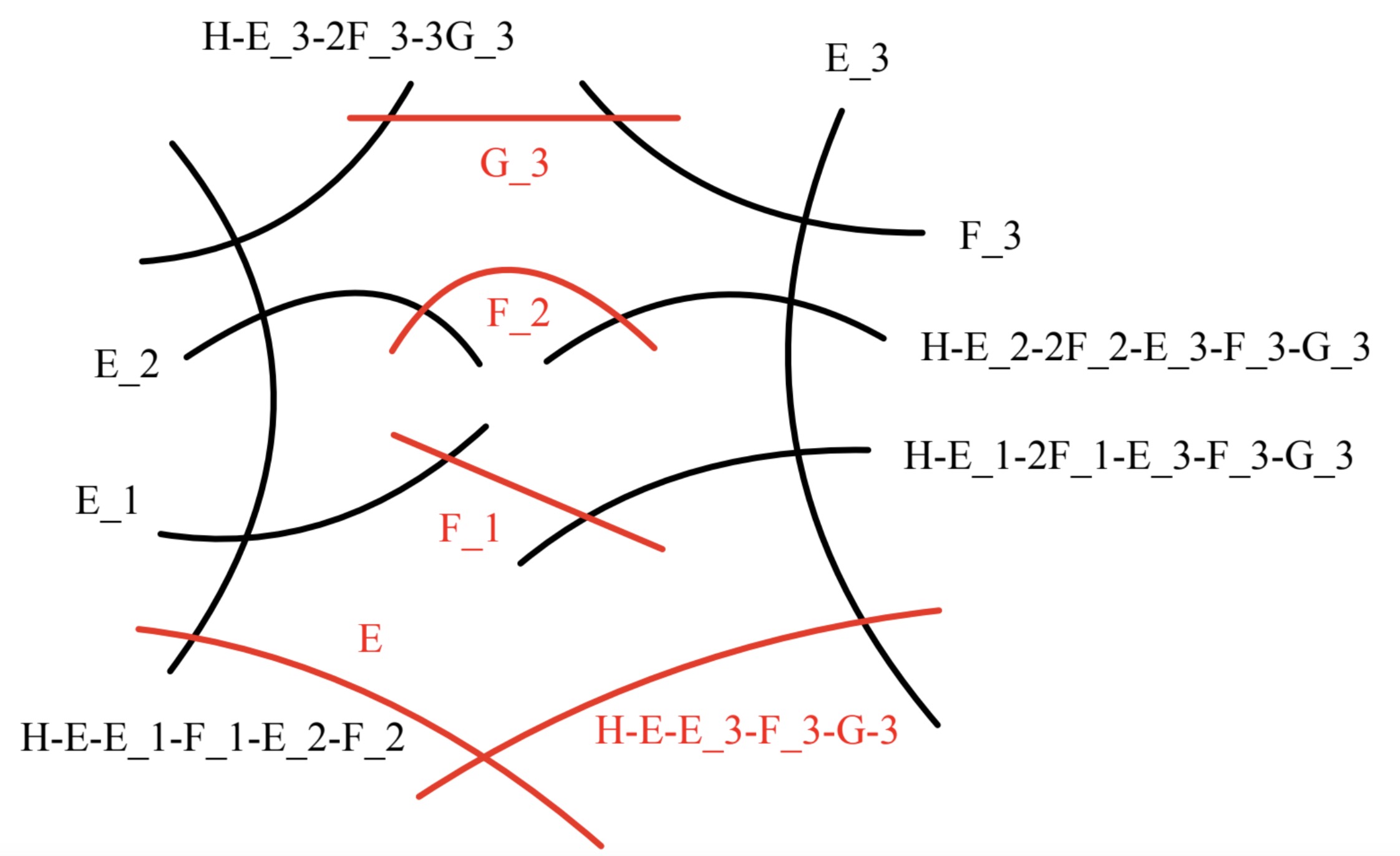} 
\caption{Configuration of the lines on $\wt{Z}$} 
\label{conf}
\end{figure}

Let $H$ be the class of the pull-back of $\mtc{O}_{\mb{P}^2}(1)$. Then the proper transform to $\wt{Z}$ of the 240 lines on $Z$ are of the following types:
\begin{enumerate}[(a)]
\item $E$, of multiplicity 24;
\item $H-E-E_3-F_3-G_3$, of multiplicity 24;
\item $F_1$, of multiplicity 64;
\item $F_2$, of multiplicity 64;
\item $G_3$, of multiplicity 64.
\end{enumerate}

Notice that the configurations of the curves in Figure \ref{conf} is symmetric: the black curves denote the (-2)-curves and the red lines denote the (-1)-curves. Using the same computation as in Section 5.1, one sees that the minimal log discrepancy of the pair $(Z,c\sum a_iL_i)$ is $1-240c>0$, when $0<c<\frac{1}{240}$. Thus the surfaces with exactly two $D_4$-singularities do not contribute any wall.

\end{proof}

\begin{remark}
\textup{Notice that the four points $p_1,...,p_4$ on $\mb{P}^1$ have a cross-ratio. This also explains why we have a one-dimensional family of surfaces with exactly two $D_4$-singularities.}
\end{remark}

\subsection{Degree one del Pezzo with an $A_8$-singularity two $\frac{1}{9}(1,2)$-singularities}

Among all K-polystable degeneration of smooth degree 1 del Pezzo surface, there is a special one $X_{\infty}$ with an $A_8$-singularity and two $\frac{1}{9}(1,2)$-singularities. The surface $X_{\infty}$ can be viewed as a degree $18$ hypersurface in $\mb{P}(1,2,9,9)$ given by the equation $z_3z_4=z_2^9$, where $(z_1:z_2:z_3:z_4)$ is the coordinates of weights $1,2,9,9$ respectively. Moreover, projecting to $\mb{P}(1,2,9)$ by $$(z_1:z_2:z_3:z_4)\mapsto (z_1:z_2:z_3+z_4)$$ realizes $X_{\infty}$ as a double cover of $\mb{P}(1,2,9)$ branched along the curve $C_{\infty}=\{v^9=w^2\}$, where $(u:v:w)$ is the coordinate of $\mb{P}(1,2,9)$. In fact, this double cover map is given by the linear system $|-2K_{X_{\infty}}|$. 

Recall that for a smooth del Pezzo surface $X_t$ of degree one, the map given by the linear system $|-2K_{X_{t}}|$ is a double cover to $\mb{P}(1,1,2)\subseteq \mb{P}^3$ branched along a sextic curve $C_t$. The 240 lines on $X_t$ are sent pairwise to the 120 conics on $\mb{P}(1,1,2)$ obtained by intersecting the $\mb{P}(1,1,2)$ with the 120 tritangent planes of $C_t$ in $\mb{P}^3$. See \cite{KRSS} for details. 

The degeneration of $\mb{P}(1,1,2)$ to $\mb{P}(1,2,9)$ can be observed in $\mb{P}^{15}$: they are embedded into $\mb{P}^{15}$ by the complete linear series $|\mtc{O}(6)|$ and $|\mtc{O}(18)|$, respectively. In particular, the 120 conics degenerate to curves in $\mb{P}(1,2,9)$ of degree 6. Notice that the pair $(\mb{P}(1,2,9),\frac{1}{2}C_{\infty})$ is a $\mb{T}$-variety of complexity one: there is a $\mb{G}_m$-action given by $$\lambda\cdot(u:v:w)=(u:\lambda^2 v:\lambda^9w).$$ In particular, each of the 120 sextic curves is $\mb{G}_m$-invariant, and hence is defined by one of the following four equations: $$u^6=0,\quad u^4v=0,\quad u^2v^2=0,\quad v^3=0.$$  However, the multiplicity of them is not clear to us. We have the following partial result.

\begin{prop}
Let $X_{\infty}$ be the degree 1 del Pezzo surface as above, and $L_i$'s be the 240 lines on it counted with multiplicities. Then the pair $(X_{\infty},c\sum L_i)$ is either K-polystable for any $0<c<\frac{1}{240}$ or K-unstable for any $0<c<\frac{1}{240}$. Moreover, it is K-polystable if and only if $$\ord_{u=0}\left(\sum l_j\right)=240\quad \textup{and}\quad \ord_{u=0}\left(\sum l_j\right)=240,$$ where $l_i$'s are the 120 sextic curves in $\mb{P}(1,2,9)$ given by the images of $L_i$'s.
\end{prop}

\begin{proof}
The main tool here we use is equivariant K-stability (cf. \cite{Zhu21}). Assume we are in the case when $\ord_{u=0}\left(\sum l_j\right)=240$ and $\ord_{u=0}\left(\sum l_j\right)=240$. By \cite[Theorem 1.2]{LZ22}, it suffices to check the K-stability of the pair $(\mb{P}(1,2,9),\frac{1}{2}C_{\infty}+c\sum l_i)$, which is a $\mb{T}$-pair of complexity one. By \cite[Theorem 1.3.9]{CA21}, we only need to compute the $\beta$-invariant of the pair with respect to all $\mb{G}_m$-invariant divisors on $\mb{P}(1,2,9)$ (see also \cite[Theorem 2.9]{zha22} for the statement for pairs). The divisor $\{u=0\}$ is the unique horizontal divisor on $X_{\infty}$. We have that $$A_{(\mb{P}(1,2,9),\frac{1}{2}C_{\infty}+c\sum l_i)}(\{u=0\})=1-240c$$ and that $$S_{(\mb{P}(1,2,9),\frac{1}{2}C_{\infty}+c\sum l_i)}(\{u=0\})=\frac{(3-720c)}{\mtc{O}(1)^2}\int_{0}^1\mtc{O}(1-t)^2dt=1-240c,$$ and thus $\beta_{(\mb{P}(1,2,9),\frac{1}{2}C_{\infty}+c\sum l_i)}(\{u=0\})=0$. For the vertical divisor $\{v=0\}$ we have that $$A_{(\mb{P}(1,2,9),\frac{1}{2}C_{\infty}+c\sum l_i)}(\{v=0\})=1-240c$$ and that $$S_{(\mb{P}(1,2,9),\frac{1}{2}C_{\infty}+c\sum l_i)}(\{v=0\})=\frac{(3-720c)}{\mtc{O}(1)^2}\int_{0}^{\frac{1}{2}}\mtc{O}(1-2t)^2dt=\frac{1}{2}(1-240c).$$ For other vertical divisors $D_t=\{v^9=tw^2\}$ with $\lambda\in\mb{C}^{*}$, similarly, we have that that $\beta_{(\mb{P}(1,2,9),\frac{1}{2}C_{\infty}+c\sum l_i)}(D_t)>0$. This conclude for the case $\ord_{u=0}\left(\sum l_j\right)=\ord_{u=0}\left(\sum l_j\right)=240$. For the other case, the $\beta$-invariant for the horizontal divisor $\{u=0\}$ is non-zero when $0<c<\frac{1}{240}$ so the pair is K-unstable.

\end{proof}

\bibliographystyle{alpha}
\bibliography{citation}

\end{document}